\documentclass[11pt,a4paper]{article}

\usepackage[margin=27mm]{geometry}
\usepackage{amsmath,amssymb,mathtools,amsthm}
\usepackage{microtype}
\usepackage[hidelinks]{hyperref}
\numberwithin{equation}{section}

\newtheorem{theorem}{Theorem}[section]
\newtheorem{proposition}[theorem]{Proposition}
\newtheorem{lemma}[theorem]{Lemma}
\newtheorem{corollary}[theorem]{Corollary}
\theoremstyle{definition}
\newtheorem{definition}[theorem]{Definition}
\newtheorem{remark}[theorem]{Remark}

\newcommand{\E}{\mathbb E}
\newcommand{\Prob}{\mathbb P}
\newcommand{\F}{\mathbb F}
\newcommand{\ind}{\mathbf 1}
\newcommand{\cA}{\mathcal A}

\title{Bounded independence for the inverse star discrepancy}
\author{Kosuke Suzuki\thanks{The work of K.S.\ is supported by JSPS KAKENHI
Grant Numbers 24K06857 and 26K00620.}}
\date{}

\begin{document}
\maketitle

\begin{abstract}
We give a random-bit-efficient construction for the inverse star discrepancy.
For every fixed $u\in(0,1)$, $k$-wise independent uniform points
$\boldsymbol{X}_1,\ldots,\boldsymbol{X}_N$ with
$k=O(d(1+\log(1+N/d)))$ satisfy the Monte Carlo bound
$D_N^*(\boldsymbol{X}_1,\ldots,\boldsymbol{X}_N)
=O(\sqrt{d/N})$ with probability at least $u$.  Consequently,
$N=O(d\varepsilon^{-2})$ and
$k=O(d(1+\log\varepsilon^{-1}))$ suffice to attain discrepancy at most
$\varepsilon$.  The proof isolates the finitely many moments required by a
chaining argument and gives explicit constants.  A random vector-valued
polynomial over a finite field realizes the required bounded independence on
a grid using
$O(d^2(1+\log(1+N/d))\log N)$ random bits, rather than the
$\Theta(dN\log(dN))$ bits used by independent grid sampling.
\end{abstract}

\medskip
\noindent\textbf{Keywords.}
Star discrepancy, inverse star discrepancy, $k$-wise independence,
random-bit complexity, finite-field polynomial construction.

\noindent\textbf{2020 Mathematics Subject Classification.}
11K38, 65C05, 65C10.

\section{Introduction}

For an $N$-point multiset
$P=(\boldsymbol{x}_1,\ldots,\boldsymbol{x}_N)$ in $[0,1)^d$, its star
discrepancy is
\begin{equation}
D_N^*(P)
=
\sup_{\boldsymbol{z}\in[0,1]^d}
\left|
\frac1N\sum_{n=1}^N
\ind_{[\boldsymbol{0},\boldsymbol{z})}(\boldsymbol{x}_n)
-
\prod_{j=1}^d z_j
\right|.
\label{eq:star-discrepancy}
\end{equation}
Here
$[\boldsymbol{0},\boldsymbol{z})=\prod_{j=1}^d[0,z_j)$.
All point collections below are equal-weight multisets.
For every function $f$ of bounded Hardy--Krause variation, the
Koksma--Hlawka inequality gives
\[
\left|
\frac1N\sum_{n=1}^N f(\boldsymbol{x}_n)
-\int_{[0,1]^d}f(\boldsymbol{x})\,d\boldsymbol{x}
\right|
\le V_{\mathrm{HK}}(f)D_N^*(P).
\]
For fixed $d$, classical low-discrepancy constructions satisfy
$D_N^*(P)=O_d(N^{-1}(\log N)^{d-1})$.  The Koksma--Hlawka inequality then
gives the deterministic QMC integration error bound of the same order; see
\cite{Nie92,DP10}.  This asymptotic regime, however, hides the dependence on
$d$ in the implied constant.

Set
\[
D^*(N,d):=\inf_{P\in([0,1)^d)^N}D_N^*(P),
\qquad
n^*(d,\varepsilon):=\min\{N\in\mathbb N:D^*(N,d)\le\varepsilon\}.
\]
The inverse problem for star discrepancy asks for the joint dependence of
$n^*(d,\varepsilon)$ on $d$ and $\varepsilon$.  For sufficiently small
$\varepsilon$, the known uniform estimates are
\[
d\varepsilon^{-1}\lesssim n^*(d,\varepsilon)
\lesssim d\varepsilon^{-2};
\]
the upper bound is due to Heinrich--Novak--Wasilkowski--Wo\'zniakowski
\cite{HNWW01}, and the lower bound to Hinrichs~\cite{Hin04}.  Dick~\cite{Dic26}
recently proved that the exponents $1$ of $d$ and $2$ of
$\varepsilon^{-1}$ in uniform polynomial tractability bounds are individually
optimal, while the exact joint dependence remains open.
The upper bound is probabilistic and was refined through covering and chaining
arguments in \cite{DGS05,Gne08,Ais11,AH14}.  In particular, if
$\boldsymbol{X}_1,\ldots,\boldsymbol{X}_N$ are independent and uniform on
$[0,1)^d$, then
\[
D_N^*(\boldsymbol{X}_1,\ldots,\boldsymbol{X}_N)
\ll \sqrt{d/N}
\]
with probability bounded below by a positive absolute constant.  Its expected
discrepancy has the matching order for $d\le N$~\cite{Doe14}.  The standard
i.i.d.\ sampler, however, uses $dN$ independent coordinates.  After
discretization at the resolution required to preserve the discrepancy order,
its random-bit cost is of order $dN\log(dN)$.  We ask how much of this
independence is actually used by the discrepancy argument.

Writing $\log_+x:=\max\{0,\log x\}$, our main result is that, for every fixed
success probability $u\in(0,1)$, $k$-wise independent uniform points
$\boldsymbol{X}_1,\ldots,\boldsymbol{X}_N$ with
$k=O(d(1+\log_+(N/d)))$ satisfy
\[
D_N^*(\boldsymbol{X}_1,\ldots,\boldsymbol{X}_N)
\ll\sqrt{d/N}
\]
with probability at least $u$.  Thus full independence of the $N$ points is unnecessary even when
$N$ is much larger than $d$.  Theorem~\ref{thm:finite} gives an explicit
constant and the full dependence on $u$.

The proof transfers the chaining argument by matching only the even moments it uses.  An
$r$th moment of a centered Bernoulli sum is unchanged under $r$-wise
independence, so choosing $r$ according to the covering entropy at each level
recovers the independent-point bounds.  This principle is classical in
randomness-efficient sampling and limited-independence concentration
\cite{BR94,SSS95}; see also the sharp bounds in \cite{Sko22}.

The required sample space is realized by the classical random-polynomial
construction over a finite field~\cite{Jof74}.  After lifting the evaluations
to a grid, it gives the same discrepancy order with
\[
O\!\left(d^2\left(1+\log_+\frac Nd\right)\log N\right)
\]
random bits for fixed $u$.  Consequently, at
$N\asymp d\varepsilon^{-2}$ the random-bit cost depends only logarithmically
on $\varepsilon^{-1}$, in contrast to independent grid sampling.

Restricted-randomness Monte Carlo integration was studied in
\cite{HNP04}.  In computational geometry, bounded independence was used for
$\varepsilon$-approximations, $\varepsilon$-nets, and discrepancy-type
problems in \cite{GR97,MRS01}.  These results are formulated for fixed
VC-exponent and do not yield the dimension-uniform estimate above.
G\'omez--G\'omez-P\'erez--Pillichshammer~\cite{GGP21} obtained the order
$\sqrt{d/N}$ from computationally secure pseudorandom generators, rather than
from exact bounded independence.  The polynomial- and rank-one-lattice unions
in \cite{DP26,DD26} have short finite descriptions but currently achieve the
weaker bound $O(d\log N/\sqrt N)$.  To the best of our knowledge, the
combination of $N=O(d\varepsilon^{-2})$ with
$k=O(d(1+\log\varepsilon^{-1}))$-wise independence has not previously been
stated explicitly.

Sections~2 and~3 establish the finite-moment transfer, and Section~4 gives
the finite-field realization and its randomness cost.

\section{Bounded independence and a finite-moment estimate}

\begin{definition}
Random vectors $\boldsymbol{X}_1,\ldots,\boldsymbol{X}_N$ are called
\emph{$k$-wise independent} if
every subfamily of at most $k$ distinct vectors is independent.
\end{definition}

Throughout Section~\ref{sec:finite-general}, every $\boldsymbol{X}_n$ has the uniform
distribution on $[0,1)^d$.  Thus the coordinates inside one point are also
independent and uniform, although the full family of points may be strongly
dependent.  The discrete points used in Section~\ref{sec:finite-field} are
introduced separately there.

For every integer $r\ge2$, put
\[
a_r:=\frac{e(r!)^{1/r}}r.
\]
We use the following finite-moment concentration estimate.  It combines exact
moment matching with the bounded-variable Bernstein and Hoeffding
exponential-moment bounds.

\begin{lemma}[Finite-moment Bernstein and Hoeffding bounds]
\label{lem:finite-moment}
Let $Z_1,\ldots,Z_M$ be $k$-wise independent Bernoulli random variables with
common mean $p$.  Put
\[
S=\sum_{i=1}^M(Z_i-p),
\qquad
V=Mp(1-p).
\]
If $2\le r\le k$ is even, then
\begin{align}
\Prob\left[
 |S|>
 a_r\left(\sqrt{rV}+\frac r3\right)
\right]
&\le e^{-r/2},
\label{eq:moment-tail-bernstein}\\
\Prob\left[
 |S|>
 \frac{a_r}{2}\sqrt{rM}
\right]
&\le e^{-r/2}.
\label{eq:moment-tail-hoeffding}
\end{align}
\end{lemma}

\begin{proof}
Since $r$ is even and $r\le k$, the moment
$\E[|S|^r]=\E[S^r]$ agrees with the fully independent case.  We may therefore
assume independence when estimating it.  For every $t>0$ and $x\in\mathbb R$,
\begin{equation}
|x|^r
\le
\frac{r!}{t^r}
\sum_{j=0}^{\infty}\frac{(tx)^{2j}}{(2j)!}
=
\frac{r!}{2t^r}\bigl(e^{tx}+e^{-tx}\bigr).
\label{eq:cosh-moment}
\end{equation}

For the Bernstein estimate, put $X_i=Z_i-p$.  Then
$\E[X_i]=0$, $X_i\le1$, and
\[
\sum_{i=1}^M\E[X_i^2]=V.
\]
The same assertions needed for the upper-tail bound hold for $-X_i$.
Hence the standard Bernstein exponential-moment bound
\cite[Section~2]{BLM13}
gives
\[
\E[e^{\pm tS}]
\le
\exp\left(\frac{Vt^2}{2(1-t/3)}\right),
\qquad 0<t<3.
\]
If $V=0$, the Bernstein claim is trivial.  Otherwise, combining
\eqref{eq:cosh-moment} with this bound and choosing
$t=(\sqrt{V/r}+1/3)^{-1}$ gives
\[
\E[|S|^r]
\le
\frac{r!}{t^r}
\exp\left(\frac{Vt^2}{2(1-t/3)}\right)
\le
\frac{r!}{t^r}e^{r/2}
=
e^{-r/2}
\left(
 a_r\left(\sqrt{rV}+\frac r3\right)
\right)^r.
\]
Therefore Markov's inequality gives
\[
\Prob\left[
 |S|>
 a_r\left(\sqrt{rV}+\frac r3\right)
\right]
= \Prob\left[
 |S|^r>
 \left(a_r\left(\sqrt{rV}+\frac r3\right)\right)^r
\right]
\le e^{-r/2}.
\]

For the Hoeffding estimate, each $X_i=Z_i-p$ lies in the interval
$[-p,1-p]$, which has length one.  Hence Hoeffding's lemma and independence
give
\[
\E[e^{\pm tS}]
=
\prod_{i=1}^M\E[e^{\pm tX_i}]
\le e^{Mt^2/8}.
\]
Combining this with \eqref{eq:cosh-moment} and taking
$t=2\sqrt{r/M}$ gives
\[
\E[|S|^r]
\le
\frac{r!}{t^r}e^{Mt^2/8}
=
e^{-r/2}
\left(
 \frac{a_r}{2}\sqrt{rM}
\right)^r.
\]
Markov's inequality now gives \eqref{eq:moment-tail-hoeffding}.
\end{proof}

\section{Bounded-independence transfer of the chaining argument}
\label{sec:finite-general}

In this section, we follow the dyadic chaining
formulation as in \cite{Suz26}.
For integers $1\le\mu\le H$, put
\begin{equation}
C_m=
\begin{cases}
\min\{\mathcal N_{[]}(d,2^{-m}),\mathcal N(d,2^{-m})\},
&\mu\le m<H,\\
\mathcal N(d,2^{-H}),&m=H.
\end{cases}
\label{eq:chain-cardinality}
\end{equation}
Here $\mathcal N_{[]}(d,\delta)$ and $\mathcal N(d,\delta)$ denote,
respectively, the bracketing and direct covering numbers for anchored boxes.  We use the standard relation
$\mathcal N(d,\delta)\le2\mathcal N_{[]}(d,\delta)$; see, for example,
\cite{Gne08}.

Our starting point is the corresponding estimate for independent Monte Carlo
points.  Theorem~7.2 of \cite{Suz26} states that, for positive weights satisfying
$w_{\mathcal B}+\sum_{h=\mu+1}^{H}w_h<1$, there exists an $N$-point set
$P\subset[0,1)^d$ such that
\begin{align}
ND_N^*(P)
&\le
\sqrt{\frac N2\log\frac{2C_\mu}{w_{\mathcal B}}}
\notag\\
&\quad+
\sum_{h=\mu+1}^{H}
\left(
\sqrt{2N2^{-(h-1)}\log\frac{2C_h}{w_h}}
+\frac13\log\frac{2C_h}{w_h}
\right)
+N2^{-H}.
\label{eq:independent-chain}
\end{align}
Its proof applies Hoeffding's inequality to the coarse boxes and Bernstein's
inequality to the shells, followed by a union bound over the chaining
families.

These concentration inequalities use the factorization of exponential
moments and therefore cannot be applied directly when the points are only
$k$-wise independent.  On the other hand, every even moment of order at most
$k$ of each centered counting sum agrees with the fully independent case.
Lemma~\ref{lem:finite-moment} converts this moment matching into finite-moment
analogues of the Hoeffding and Bernstein bounds.

The standard argument of dyadic chaining produces finite families
$\mathcal B_\mu$ and $\cA_h$, $\mu<h\le H$, such that
\begin{equation}
|\mathcal B_\mu|\le C_\mu,
\qquad
|\cA_h|\le C_h,
\qquad
\lambda(A)\le2^{-(h-1)}
\quad(A\in\cA_h),
\label{eq:chain-properties}
\end{equation}
and, deterministically,
\begin{align}
N D_N^*(\boldsymbol{X}_1,\ldots,\boldsymbol{X}_N)
&\le
\max_{B\in\mathcal B_\mu}
\left|\sum_{n=1}^N
\bigl(\ind_B(\boldsymbol{X}_n)-\lambda(B)\bigr)\right|
\notag\\
&\quad+
\sum_{h=\mu+1}^{H}
\max_{A\in\cA_h}
\left|\sum_{n=1}^N
\bigl(\ind_A(\boldsymbol{X}_n)-\lambda(A)\bigr)\right|
+N2^{-H}.
\label{eq:finite-chain}
\end{align}
see \cite[Proof of Theorem~7.2]{Suz26}.
Thus only the concentration estimate for each fixed member of these families
must be reconsidered.

\begin{theorem}[Bounded-independence transfer]
\label{thm:limited-chaining}
Let $d,N,H\ge1$ and $1\le\mu\le H$ be integers, and let $u\in(0,1)$.
Choose positive weights $w_{\mathcal B}$ and $w_h$, $\mu<h\le H$, such that
\[
w_{\mathcal B}+\sum_{h=\mu+1}^{H}w_h\le1.
\]
Put
\[
L_{\mathcal B}
=
\log\frac{C_\mu}{(1-u)w_{\mathcal B}},
\qquad
L_h
=
\log\frac{C_h}{(1-u)w_h},
\qquad \mu<h\le H,
\]
and choose even integers
\[
r_{\mathcal B}\ge2L_{\mathcal B},
\qquad
r_h\ge2L_h,
\qquad \mu<h\le H.
\]
Finally, put
\[
k
=
\min\left\{
N,\,
\max\bigl(\{r_{\mathcal B}\}\cup\{r_h:\mu<h\le H\}\bigr)
\right\}.
\]
If $\boldsymbol{X}_1,\ldots,\boldsymbol{X}_N$ are $k$-wise independent
uniform points in $[0,1)^d$, then, with probability at least $u$,
\begin{align}
N D_N^*(\boldsymbol{X}_1,\ldots,\boldsymbol{X}_N)
&\le
\frac{a_{r_{\mathcal B}}}{2}\sqrt{r_{\mathcal B}N}
+\sum_{h=\mu+1}^{H}
a_{r_h}
\left(
\sqrt{r_hN2^{-(h-1)}}+\frac{r_h}{3}
\right)
+N2^{-H}.
\label{eq:limited-chain}
\end{align}
\end{theorem}

\begin{proof}
For every fixed $B\in\mathcal B_\mu$ or $A\in\cA_h$, the corresponding
indicators of the points are $k$-wise independent Bernoulli variables with
mean $\lambda(B)$ or $\lambda(A)$, respectively.  First suppose that $k<N$.
Every selected moment order is then at most $k$, so
Lemma~\ref{lem:finite-moment} applies.  If $k=N$, the $N$ points are fully
independent and the same moment estimates are available at every even order.
Consequently,
\[
\Prob\left[
\left|\sum_{n=1}^N
\bigl(\ind_B(\boldsymbol{X}_n)-\lambda(B)\bigr)\right|
>
\frac{a_{r_{\mathcal B}}}{2}\sqrt{r_{\mathcal B}N}
\right]
\le e^{-r_{\mathcal B}/2}
\le e^{-L_{\mathcal B}}.
\]
For $A\in\cA_h$, \eqref{eq:chain-properties} gives
\[
N\lambda(A)(1-\lambda(A))\le N2^{-(h-1)},
\]
and hence
\[
\Prob\left[
\left|\sum_{n=1}^N
\bigl(\ind_A(\boldsymbol{X}_n)-\lambda(A)\bigr)\right|
>
a_{r_h}\left(
\sqrt{r_hN2^{-(h-1)}}+\frac{r_h}{3}
\right)
\right]
\le e^{-r_h/2}
\le e^{-L_h}.
\]
Using \eqref{eq:chain-properties}, the union bound gives total failure
probability at most
\begin{align*}
|\mathcal B_\mu|e^{-L_{\mathcal B}}
+\sum_{h=\mu+1}^{H}|\cA_h|e^{-L_h}
&\le
C_\mu e^{-L_{\mathcal B}}
+\sum_{h=\mu+1}^{H}C_he^{-L_h}\\
&=
(1-u)\left(
w_{\mathcal B}+\sum_{h=\mu+1}^{H}w_h
\right)
\le1-u.
\end{align*}
Combining these estimates with \eqref{eq:finite-chain} proves the claim.
\end{proof}

Theorem~2.9 of \cite{Gne24} and the elementary bound
$d^d/d!\le e^d$ imply
$\mathcal N_{[]}(d,\delta)\le e^d\delta^{-d}$; the case $d=1$ is
immediate.  Together with
$\mathcal N(d,\delta)\le2\mathcal N_{[]}(d,\delta)$, this gives
\begin{equation}
\log C_m\le d(1+m\log2)+\log2,
\qquad 1\le m\le H.
\label{eq:cover-entropy}
\end{equation}

\begin{theorem}
\label{thm:finite}
Let $d,N\ge1$ and $u\in(0,1)$, and put
\[
D=d+\log((1-u)^{-1})+2.
\]
Define
\begin{equation}
H=
\max\left\{
12,\,
\left\lceil
\frac12\log_2\frac ND
\right\rceil+2
\right\}
\label{eq:finite-H}
\end{equation}
and
\begin{equation}
k_0
=
\min\left\{
N,
2\left\lceil
(d+1)(1+H\log2)
+\log((1-u)^{-1})
\right\rceil
\right\}.
\label{eq:kfinite}
\end{equation}
If $\boldsymbol{X}_1,\ldots,\boldsymbol{X}_N$ are $k_0$-wise independent
uniform points in $[0,1)^d$, then
\begin{equation}
\Prob\left[
D_N^*(\boldsymbol{X}_1,\ldots,\boldsymbol{X}_N)
\le \frac72\sqrt{\frac DN}
\right]
\ge u.
\label{eq:finite-main}
\end{equation}
\end{theorem}

\begin{proof}
If $N\le D$, the conclusion is automatic from
$D_N^*(\boldsymbol{X}_1,\ldots,\boldsymbol{X}_N)\le1$.  Assume otherwise and apply
Theorem~\ref{thm:limited-chaining} with
\[
\mu=12, \qquad
w_{\mathcal B}=\frac12,
\qquad
w_h=2^{11-h} \quad(13 \le h \le H).
\]
These weights are admissible because
\[
w_{\mathcal B}+\sum_{h=13}^{H}w_h
<
\frac12+\sum_{h=13}^{\infty}2^{11-h}
=1.
\]
Choose
\begin{align*}
r_{\mathcal B}
&=
2\left\lceil
d(1+12\log2)+\log\frac4{1-u}
\right\rceil
\le
2D(1+12\log2),\\
r_h
&=
2\left\lceil
d(1+h\log2)
+\log\frac{2^{h-10}}{1-u}
\right\rceil
\le
2D(1+h\log2),
\qquad 13\le h\le H,
\end{align*}
These definitions also give
\[
r_{\mathcal B},r_h\ge 2\lceil1+12\log2\rceil=20.
\]
By \eqref{eq:cover-entropy}, these even orders dominate those required in
Theorem~\ref{thm:limited-chaining}.  Since $h\le H$ and $H\ge12$,
\begin{align*}
d(1+12\log2)+\log\frac4{1-u}
&\le
(d+1)(1+H\log2)+\log((1-u)^{-1}),\\
d(1+h\log2)+\log\frac{2^{h-10}}{1-u}
&\le
(d+1)(1+H\log2)+\log((1-u)^{-1}).
\end{align*}
The monotonicity of the ceiling function therefore shows that all the chosen
orders are at most
\[
2\left\lceil
(d+1)(1+H\log2)
+\log((1-u)^{-1})
\right\rceil.
\]
Hence the required independence order is at most $k_0$.
Theorem~\ref{thm:limited-chaining} now gives \eqref{eq:limited-chain} with
probability at least $u$.

Robbins' form of Stirling's estimate~\cite{Rob55} yields, for every integer
$r\ge2$,
\[
a_r
\le(2\pi r)^{1/(2r)}e^{1/(12r^2)}.
\]
The logarithm of the right-hand side has derivative
\[
\frac{1-\log(2\pi r)}{2r^2}-\frac{1}{6r^3}<0
\qquad (r\ge1).
\]
Therefore, for each moment order used above, evaluation at $r=20$ gives
\[
a_r
\le(40\pi)^{1/40}e^{1/(12\cdot20^2)}
<\frac65.
\]

We now estimate the three terms on the right-hand side of
\eqref{eq:limited-chain}.  For the first term, $\log2<3/4$ gives
\[
\frac{a_{r_{\mathcal B}}}{2}\sqrt{r_{\mathcal B}N}
\le
\frac35\sqrt{2(1+12\log2)}\,\sqrt{ND}
<
\frac{6\sqrt5}{5}\sqrt{ND}.
\]

The second term consists of a square-root part and a linear part.  If $H=12$,
it is empty, so assume henceforth that $H\ge13$.
$\log2<3/4$ and $h\ge13$ give
$1+h\log2\le h$.  The ratio of consecutive terms of
$\sqrt{2h}\,2^{-(h-1)/2}$ is at most
$\sqrt{7/13}<3/4$.  Hence
\[
\begin{aligned}
\sum_{h=13}^H
a_{r_h}\sqrt{r_hN2^{-(h-1)}}
&\le
\frac65\sqrt{ND}
\sum_{h=13}^{\infty}
\sqrt{2h}\,2^{-(h-1)/2}\\
&\le
\frac65\sqrt{ND}
\sum_{h=13}^{\infty}
\sqrt{26}\,2^{-6}\left(\frac34\right)^{h-13}\\
&=
\frac{3\sqrt{26}}{40}\sqrt{ND}
<
\frac25\sqrt{ND}.
\end{aligned}
\]

For the linear part, since $H\ge13$, the definition of $H$ gives
\[
H=\left\lceil\frac12\log_2\frac ND\right\rceil+2.
\]
Since $\lceil x\rceil-1<x$, it follows that
\[
\frac12\log_2\frac ND>H-3,
\qquad
\sqrt{\frac ND}>2^{H-3}.
\]
Using again $a_{r_h}<6/5$ and $r_h\le2D(1+h\log2)$, we obtain
\[
\begin{aligned}
\sum_{h=13}^H\frac{a_{r_h}r_h}{3}
\le
\frac45D\sum_{h=13}^{H}(1+h\log2)
<
\frac{32}{5}\frac{H^2}{2^H}\sqrt{ND}
\le
\frac{169}{1280}\sqrt{ND}
<
\frac2{15}\sqrt{ND}.
\end{aligned}
\]
Here $H^2/2^H$ is decreasing for integers $H\ge3$.  Finally,
\[
N2^{-H}
=
\sqrt{ND}\left(\sqrt{\frac ND}\,2^{-H}\right)
\le\frac14\sqrt{ND}.
\]
Substitution into \eqref{eq:limited-chain} gives
\[
N D_N^*(\boldsymbol{X}_1,\ldots,\boldsymbol{X}_N)
<
\left(
\frac{6\sqrt5}{5}+\frac25+\frac2{15}+\frac14
\right)\sqrt{ND}
<
\frac72\sqrt{ND},
\]
which proves \eqref{eq:finite-main}.
\end{proof}

\begin{corollary}[Randomized construction for the inverse star discrepancy]
\label{cor:inverse}
For all $d\ge1$, $0<\varepsilon\le1$, and $u\in(0,1)$, it is sufficient to take
\[
N=
\left\lceil
\frac{49}{4}
\bigl(d+\log((1-u)^{-1})+2\bigr)\varepsilon^{-2}
\right\rceil
\]
and
\[
k=O\!\left(
 d(1+\log\varepsilon^{-1})+\log((1-u)^{-1})
\right)
\]
to obtain $k$-wise independent uniform points
$\boldsymbol{X}_1,\ldots,\boldsymbol{X}_N$ satisfying
\[
D_N^*(\boldsymbol{X}_1,\ldots,\boldsymbol{X}_N)\le\varepsilon
\]
with probability at least $u$.  For fixed $u$, this gives
$N=O(d\varepsilon^{-2})$ and
$k=O(d(1+\log\varepsilon^{-1}))$.
\end{corollary}

\begin{proof}
Apply Theorem~\ref{thm:finite} with the displayed value of $N$.
Then $H=O(1+\log\varepsilon^{-1})$, and substitution into
\eqref{eq:kfinite} gives the asserted independence order.
\end{proof}

\section{Algebraic realization over a finite field}
\label{sec:finite-field}

\begin{definition}[Finite-field polynomial construction]
\label{def:finite-field-construction}
Let $d,N\ge1$ and $1\le k\le N$.  Let $q\ge N$ be a prime power, choose
distinct elements
$\xi_1,\ldots,\xi_N\in\F_q$, and fix a bijection
\[
\iota:\F_q\longrightarrow\{0,\ldots,q-1\}.
\]
Choose $\boldsymbol{A}_0,\ldots,\boldsymbol{A}_{k-1}$ independently and
uniformly from $\F_q^d$, and form the random vector-valued polynomial
\begin{equation}
\boldsymbol{F}(T)
=\boldsymbol{A}_0+\boldsymbol{A}_1T+\cdots
+\boldsymbol{A}_{k-1}T^{k-1}.
\label{eq:random-poly}
\end{equation}
For $1\le n\le N$, set
\[
\boldsymbol{Y}_n=\boldsymbol{F}(\xi_n)\in\F_q^d
\]
and define the grid point
\begin{equation}
\boldsymbol{X}_n=(X_{n,1},\ldots,X_{n,d}),
\qquad
X_{n,j}
=
\frac{\iota(Y_{n,j})}{q},
\qquad 1\le j\le d.
\label{eq:grid-lift}
\end{equation}
The resulting ordered random family
$(\boldsymbol{X}_1,\ldots,\boldsymbol{X}_N)$ is called the finite-field
polynomial construction with independence order $k$.
\end{definition}

\begin{proposition}[Polynomial evaluation sample space]
\label{prop:poly-independence}
In Definition~\ref{def:finite-field-construction}, the random vectors
$\boldsymbol{Y}_1,\ldots,\boldsymbol{Y}_N$ are $k$-wise independent and
uniform on $\F_q^d$.  Consequently,
$\boldsymbol{X}_1,\ldots,\boldsymbol{X}_N$ are $k$-wise independent and
uniform on the grid $\{0,1/q,\ldots,(q-1)/q\}^d$.
\end{proposition}

\begin{proof}
For distinct $n_1,\ldots,n_\ell$ with $\ell\le k$, the evaluation map
\[
(\boldsymbol{A}_0,\ldots,\boldsymbol{A}_{k-1})
\longmapsto
\bigl(\boldsymbol{F}(\xi_{n_1}),\ldots,
\boldsymbol{F}(\xi_{n_\ell})\bigr)
\]
is surjective.  Indeed, its first $\ell$ coefficient columns form an invertible
Vandermonde matrix.  The image of the uniform coefficient vector is therefore
uniform on $(\F_q^d)^\ell$.  The assertion for the grid points follows from
the bijectivity of $\iota$.
\end{proof}

\begin{theorem}[Star discrepancy of the finite-field construction]
\label{thm:finite-algebraic}
Let $d,N\ge1$ and $u\in(0,1)$, and let $k_0$ be the independence
order defined in Theorem~\ref{thm:finite}.  For every prime power $q\ge N$,
the construction in Definition~\ref{def:finite-field-construction} with
$k=k_0$ generates an $N$-point multiset satisfying
\[
D_N^*(\boldsymbol{X}_1,\ldots,\boldsymbol{X}_N)
\le
\frac92
\sqrt{\frac{d+\log((1-u)^{-1})+2}{N}}
\]
with probability at least $u$.
\end{theorem}

\begin{proof}
If $N\le d+\log((1-u)^{-1})+2$, the claim follows from
$D_N^*(\boldsymbol{X}_1,\ldots,\boldsymbol{X}_N)\le1$.
Assume henceforth that $N>d+\log((1-u)^{-1})+2$.
Introduce auxiliary random vectors
$\boldsymbol{U}_1,\ldots,\boldsymbol{U}_N$, independent and uniform
on $[0,1)^d$ and independent of the polynomial coefficients, and put
\[
\widetilde{\boldsymbol{X}}_n
=(\widetilde X_{n,1},\ldots,\widetilde X_{n,d}),
\qquad
\widetilde X_{n,j}
=
\frac{\iota(Y_{n,j})+U_{n,j}}q.
\]
By Proposition~\ref{prop:poly-independence}, the points
$\widetilde{\boldsymbol{X}}_1,\ldots,\widetilde{\boldsymbol{X}}_N$ are
$k_0$-wise independent and uniform
on $[0,1)^d$.  Theorem~\ref{thm:finite} therefore applies to them.

For every $\boldsymbol{t}=(t_1,\ldots,t_d)\in\{0,\ldots,q\}^d$, the points
$\boldsymbol{X}_n$ and $\widetilde{\boldsymbol{X}}_n$ lie in the same
$q$-grid cells, and hence
\[
\ind_{[\boldsymbol{0},\boldsymbol{t}/q)}(\boldsymbol{X}_n)
=\ind_{[\boldsymbol{0},\boldsymbol{t}/q)}(\widetilde{\boldsymbol{X}}_n).
\]
For every realization of the jitters, the maximum below is bounded by
$D_N^*(\widetilde{\boldsymbol{X}}_1,\ldots,
\widetilde{\boldsymbol{X}}_N)$.  Hence Theorem~\ref{thm:finite}
and marginalization over the jitters show that, with probability at least
$u$ over the polynomial coefficients,
\begin{equation}
\max_{\boldsymbol{t}\in\{0,\ldots,q\}^d}
\left|
\frac1N\sum_{n=1}^N
\ind_{[\boldsymbol{0},\boldsymbol{t}/q)}(\boldsymbol{X}_n)
-\prod_{j=1}^d\frac{t_j}{q}
\right|
\le\frac72\sqrt{\frac{d+\log((1-u)^{-1})+2}{N}}.
\label{eq:grid-bound}
\end{equation}

The standard grid discretization argument for star discrepancy
\cite[Proposition~3.17]{DP10} gives
\begin{equation}
D_N^*(\boldsymbol{X}_1,\ldots,\boldsymbol{X}_N)
\le
\max_{\boldsymbol{t}\in\{0,\ldots,q\}^d}
\left|
\frac1N\sum_{n=1}^N
\ind_{[\boldsymbol{0},\boldsymbol{t}/q)}(\boldsymbol{X}_n)
-\prod_{j=1}^d\frac{t_j}{q}
\right|
+\frac{d}{q}.
\label{eq:grid-discretization}
\end{equation}
Indeed, for a given $\boldsymbol{z}\in[0,1]^d$, take
$t_j=\lceil qz_j\rceil$.  The corresponding box counts agree, while
$0\le t_j/q-z_j<1/q$; the telescoping product inequality then gives the last
term in \eqref{eq:grid-discretization}.  Finally,
since $q\ge N>d+\log((1-u)^{-1})+2>d$,
\[
\frac{d}{q}
\le
\frac{d}{N}
\le
\sqrt{\frac{d+\log((1-u)^{-1})+2}{N}}
\]
gives the stated estimate.
\end{proof}

\begin{remark}
If one instead places each evaluation at the midpoint
$(\iota(Y_{n,j})+1/2)/q$ of its grid cell, the same proof replaces the
discretization term $d/q$ by $d/(2q)$.  The resulting constant in the theorem
is $4$ rather than $9/2$.
\end{remark}

\begin{corollary}[Random-bit cost]
\label{cor:random-bits}
Fix $u=1/2$ and $N\ge2$, and take $q$ to be the least power of two not
smaller than $N$.  The construction in Definition~\ref{def:finite-field-construction}
with $k=k_0$ uses
exactly $k_0d\log_2q$ random bits, and hence
\[
k_0d\log_2q
=
O\!\left(
d^2\left(1+\log_+\frac Nd\right)\log N
\right).
\]
For $N\asymp d\varepsilon^{-2}$, this becomes
\[
O\!\left(
d^2(1+\log\varepsilon^{-1})
(1+\log d+\log\varepsilon^{-1})
\right).
\]
\end{corollary}

\begin{proof}
There are $k_0d$ independent uniform coefficients in $\F_q$, each requiring
exactly $\log_2q$ bits.  Since $k_0\le N$, evaluation at the $N$ distinct
field elements determines each polynomial, so the coefficient space and the
space of ordered grid outputs both have cardinality $q^{k_0d}$.  The bounds
follow from $N\le q<2N$, \eqref{eq:finite-H}, and \eqref{eq:kfinite}.
\end{proof}

\section*{Declaration of generative AI use}
The author used ChatGPT 5.6 Sol for literature searches, exploratory
development of ideas, mathematical discussion, and assistance in preparing
portions of the exposition and LaTeX source. All mathematical arguments,
calculations, references, and conclusions were independently checked and
verified by the author, who takes full responsibility for the contents of
the paper.

\end{document}